\documentclass[letterpaper]{amsart}
\usepackage{amsmath,amssymb,amsfonts,amsthm,mathtools}
\usepackage{ifluatex}
\ifluatex
\usepackage{polyglossia,csquotes}
\setdefaultlanguage{english}
\usepackage{unicode-math}
\AtBeginDocument{\directlua{require("combining.lua")}}
\newcommand{{\hat \}}{\hat}
\newcommand{{\tilde \}}{\tilde}
\newcommand{{\overline \}}{\bar}
\else
\usepackage[english]{babel}
\usepackage[utf8]{inputenc}
\usepackage{csquotes}
\newcommand{\eps}{\varepsilon}
\fi
\usepackage[unicode]{hyperref}
\usepackage{graphicx,braket}
\usepackage{multirow,booktabs}
\usepackage[hypcap=true]{caption}

\usepackage[backend=biber,giveninits,autolang=langname]{biblatex}
\addbibresource{math.bib}

\AtBeginDocument{

}

\newtheorem{theorem}{Theorem}
\newtheorem*{theorem*}{Theorem}
\newtheorem{lemma}{Lemma}

\newtheorem{conjecture}{Conjecture}
\newtheorem*{problem}{Problem}
\newtheorem{corollary}{Corollary}

\theoremstyle{remark}
\newtheorem{remark}{Remark}
\newtheorem*{remark*}{Remark}

\theoremstyle{definition}
\newtheorem*{definition*}{Definition}
\newtheorem{definition}{Definition}

\DeclareMathOperator{\const}{const}

\DeclareMathOperator{\Vect}{Vect}
\DeclareMathOperator{\SPS}{SPS}

\newcommand{\wiki}[2][en]{\href{https://#1.wikipedia.org/wiki/#2}{#2}}

\title[Multiple numerical invariants]{Bifurcations of the polycycle \enquote{tears of the heart}: multiple numerical invariants}

\author{Nataliya Goncharuk}
\email{ng432@cornell.edu}
\address{%
  Cornell University\\
  Department of Mathematics\\
  310 Malott Hall\\
  Ithaca, NY 14853-4201 USA}
\thanks{Supported by RFBR project 16-01-00748-a and Laboratory Poncelet.}
\thanks{Cornell University, Department of Mathematics}
\author{Yury Kudryashov}
\email{ik333@cornell.edu}
\address{%
  Cornell University\\
  Department of Mathematics\\
  310 Malott Hall\\
  Ithaca, NY 14853-4201 USA}
\thanks{Supported by RFBR project 16-01-00748-a and Laboratory Poncelet.}
\thanks{Cornell University, Department of Mathematics}
\makeatletter
\hypersetup{
  pdftitle=\@title,
  pdfauthor=\authors,
  pdfkeywords=\@keywords,
  pdfsubject=\@subjclass
}
\makeatother
\begin{document}
\begin{abstract}
    “Tears of the heart” is a hyperbolic polycycle formed by three separatrix connections of two saddles.
    It is met in generic 3-parameter families of planar vector fields.

    In the article~\cite{IKS-th1}, it was discovered that generically, the bifurcation of a vector field with “tears of the heart” is structurally unstable.
    The authors proved that the classification of such bifurcations has a numerical invariant.

    In this article, we study the bifurcations of “tears of the heart” in more detail, and find out that the classification of such bifurcation  may have arbitrarily many numerical invariants.
\end{abstract}
\maketitle

\section{Introduction}
Due to~\cites[Theorem 3]{P59}[Theorem 2]{P62}, a~generic vector field on~the sphere~\(S^2\) is structurally stable, see definition in~\cite{AP37}.
In 1985, V.~Arnold~\cite[Sec.~I.3.2.8]{AAIS94} suggested a perspective of the development of the global bifurcation theory on the two sphere.
In particular\footnote{The rest of this paragraph, as well as large parts of \autoref{sec:prelim} are almost exact quotes from~\cite{IKS-th1}, with minor modifications. To avoid interrupting readers every time they meet an exact quote, we omit the quotation marks.}, he conjectured~\cite[Conjecture 4]{AAIS94} that a~generic \emph{finite-parameter family} of vector fields considered on the whole sphere is weakly structurally stable.
He included this conjecture in a list of six.
After the statements of these conjectures, he writes:

\blockcquote[p. 100]{AAIS94}{\itshape Certainly proofs or counterexamples to the above conjectures are necessary for investigating nonlocal bifurcations in generic \(l\)-parameter families.}

It turns out that this conjecture is wrong.
Namely,~\cite[Theorem 1]{IKS-th1} states that there exists a non-empty open subset in the space of \(3\)-parameter families of vector fields on the sphere such that each family from this set is structurally unstable.
The classification of the families from this set up to the moderate topological equivalence has a numerical invariant that can take any positive value.

Other theorems in~\cite{IKS-th1} provide us with generic families of vector fields with many numerical invariants, and even functional invariants, at the cost of higher number of parameters.
The main goal of this paper is to show that one can achieve arbitrarily many numerical invariants without increasing the number of parameters.

\section{Preliminaries}%
\label{sec:prelim}
Let us recall some definitions.
\begin{definition}
    The \emph{characteristic number} of~a~saddle is~the~absolute value of the ratio of the eigenvalues of its linearization, the negative one in the numerator.
\end{definition}
\begin{definition}
    A singular point of a vector field is called \emph{hyperbolic}, if the eigenvalues of its linearization have non-zero real parts.
\end{definition}

Denote by \(\Vect(S^2)\) the set of \(C^3\)-smooth vector fields on \(S^2\).
\begin{definition}%
    \label{def-fam-loc}
    A \emph{family of vector fields} on \(S^2\) with a base \({\mathcal B}\subset {\mathbb R}^k\) is a smooth map \(V:{\mathcal B}\to \Vect(S^2)\).
    We will also use the notation \(V=\set{v_\alpha |\alpha \in {\mathcal B}}\).
    A \emph{local} family of vector fields is a germ of a smooth map \(V:({\mathbb R}^k, 0)\to (\Vect(S^2), v_0)\).
\end{definition}
Denote by \({\mathcal V}_k(S^2)\) the space of \(k\)-parameter local families~\(V=\set{v_\alpha |\alpha \in ({\mathbb R}^k, 0)}\) of vector fields on \(S^2\) such that \(v_\alpha (x)\) is \(C^3\)-smooth in \((\alpha , x)\).

\begin{definition}%
    \label{def-oteq}
    Two vector fields \(v\) and \({\tilde v}\) on \(S^2\) are called \emph{orbitally topologically equivalent}, if there exists a~homeomorphism~\(S^2\to S^2\) that links the phase portraits of \(v\) and \({\tilde v}\),
    that is, sends orbits of~\(v\) to orbits of~\({\tilde v}\) and preserves their time orientation.
\end{definition}

The moderate equivalence was introduced in~\cite{IKS-th1} for families of vector fields with only \emph{hyperbolic singular points}, and in~\cite{GI-LBS} in the general case.
We will use the definition from~\cite{IKS-th1} since it is simpler and shorter.

For a vector field \(v\in \Vect(S^{2})\), denote by \(\SPS(v)\subset S^2\) the union of its singular points, periodic orbits and separatrices.
\begin{definition}%
    \label{def-moderate-eq}
    Two local families of vector fields \(\set{v_\alpha  | \alpha  \in  ({\mathbb R}^k, 0)}\), \(\set{{\tilde v}_{{\tilde \alpha }}| {\tilde \alpha } \in  ({\mathbb R}^k, 0)}\) on \(S^2\) with only hyperbolic singular points are called \emph{locally moderately topologically equivalent}, if there exists a germ of a map
    \begin{align}
      \label{eqn-conj}
      H&\colon ({\mathbb R}^k \times  S^2, \set{0} \times  S^2) \to  ({\mathbb R}^k \times  S^2, \set{0} \times  S^2),& H(\alpha  , x)&=(h(\alpha  ), H_\alpha  (x))
    \end{align}
    such that
    \begin{itemize}
      \item \(h\colon ({\mathbb R}^k, 0)\to ({\mathbb R}^k, 0)\) is a germ of a homeomorphism;
      \item for each \(\alpha  \in  ({\mathbb R}^k, 0)\), the map \(H_\alpha  \colon S^2 \to  S^2\) is a homeomorphism that links the phase portraits of \(v_\alpha \) and \({\tilde v}_{h(\alpha  )}\);
      \item \(H\) is continuous in \((\alpha  , x)\) at the set \(\set{0}\times \SPS(v_{0})\), and~\(H^{-1}\) is continuous in~\((\alpha , x)\) at the set~\(\set{0}\times \SPS({\tilde v}_{0})\).
    \end{itemize}
\end{definition}
See~\cite[Sec. 1.1]{IKS-th1} for a discussion of other equivalence relations on the space of families of vector fields.

\begin{remark}
    \label{rem:moderate-corr}
    The above property of \(H\) implies that if some singular points of vector fields~\(v_\alpha \) form a continuous family \(A_\alpha \), \(A_\alpha \in S^2\), then the corresponding singular points \(H_\alpha (A_\alpha )\) depend continuously on \(\alpha \).
    The same holds for limit cycles and separatrices.

    We will use this argument to enumerate singular points and separatrices of two equivalent families so that the equivalence preserves numeration.
\end{remark}

\section{Main Theorem}
\subsection{Pure existence theorem}
First, we formulate the main theorem without revealing the construction.
Given a Banach submanifold \({\mathbf M}\subset \Vect(S^2)\) of codimension~\(k\), denote by \({\mathbf M}^\pitchfork \) the set of \(k\)-parameter local families \(V\in {\mathcal V}_k(S^2)\) such that \(v_0\in {\mathbf M}\), and \(V\) is~transverse to~\({\mathbf M}\).
\begin{theorem}%
    \label{thm:main-ex}
    For each \(N\), there exist a Banach submanifold \({\mathbf{TH}}_N\subset \Vect(S^2)\) of codimension~\(3\) and a smooth surjective function~\(\Phi :{\mathbf{TH}}_N\to {\mathbb R}_+^N\) such that for two moderately topologically equivalent families \(V, {\tilde V}\in {\mathbf{TH}}_N^\pitchfork \) we have
    \begin{itemize}
      \item \(\varphi (v_0)=\varphi ({\tilde v}_0)\), where \(\varphi (v)=\Phi _1(v)+\cdots +\Phi _N(v)\);
      \item if \(\varphi (v_0)\) is irrational, then \(\Phi (v_0)=\Phi ({\tilde v}_0)\).
      \item if \(\varphi (v_0)\) is a rational number with denominator~\(q\), then for each~\(k=1,\dots ,N\) we have \(\left|\Phi _{k}(v_{0})-\Phi _{k}({\tilde v}_{0})\right|\leq \frac{2}{q}\).
    \end{itemize}
    In particular, \(\Phi \) is an invariant of the classification of the local families from the residual subset of~\({\mathbf{TH}}_N^\pitchfork \) given by \(\set{V\in {\mathbf{TH}}_N^\pitchfork |\varphi (v_0)\notin {\mathbb Q}}\) up to the moderate topological equivalence.
\end{theorem}
Using arguments similar to~\cite[Proposition 1]{IKS-th1}, one can deduce a version of \autoref{thm:main-ex} for \emph{non-local} families of vector fields.
Briefly speaking, if each of two equivalent non-local families \(V, {\tilde V}\) meets \({\mathbf{TH}}_{N}\) transversely at a unique point~\(v_{0}\), \({\tilde v}_{0}\), and is included by a small neighborhood of~\({\mathbf{TH}}_{N}\), then the linking homeomorphism~\(h\) sends \(0\) to~\(0\), and we can apply \autoref{thm:main-ex} to the germs \((V, 0)\) and \(({\tilde V}, 0)\).

\subsection{The polycycle “tear of the heart”}%
\label{sec:polycycle-tear-heart}
Now, let us describe the~submanifold~\({\mathbf{TH}}_N\subset \Vect(S^2)\).
The construction mostly repeats the definition of~\(\mathring {\mathbf{T}}\) in~\cite[Sec. 2]{IKS-th1}, with the only exception that we have more separatrices winding onto the polycycle from the outside.
Let us recall the “interesting” part of the phase portrait of a vector field~\(v\in {\mathbf{TH}}_{N}\), see \autoref{fig:pc-th3}.
For a more detailed description, including extension of the phase portrait to the whole sphere, see~\cite[Sec. 2.1]{IKS-th1}.

We say that a vector field is \emph{normalized}, if it has a unique pair of a sink and a repelling limit cycle such that there are no other singular points to the same side of the limit cycle as the sink.
We say that this sink is the infinity~\(\infty \), and consider its complement \(S^2\setminus \set{\infty }\) as the plane~\({\mathbb R}^2\).
This allows us to unambiguously speak about “exterior” and “interior” of a~closed curve on the sphere.

\begin{figure}[h]
    \centering
    \includegraphics[width=.9\textwidth]{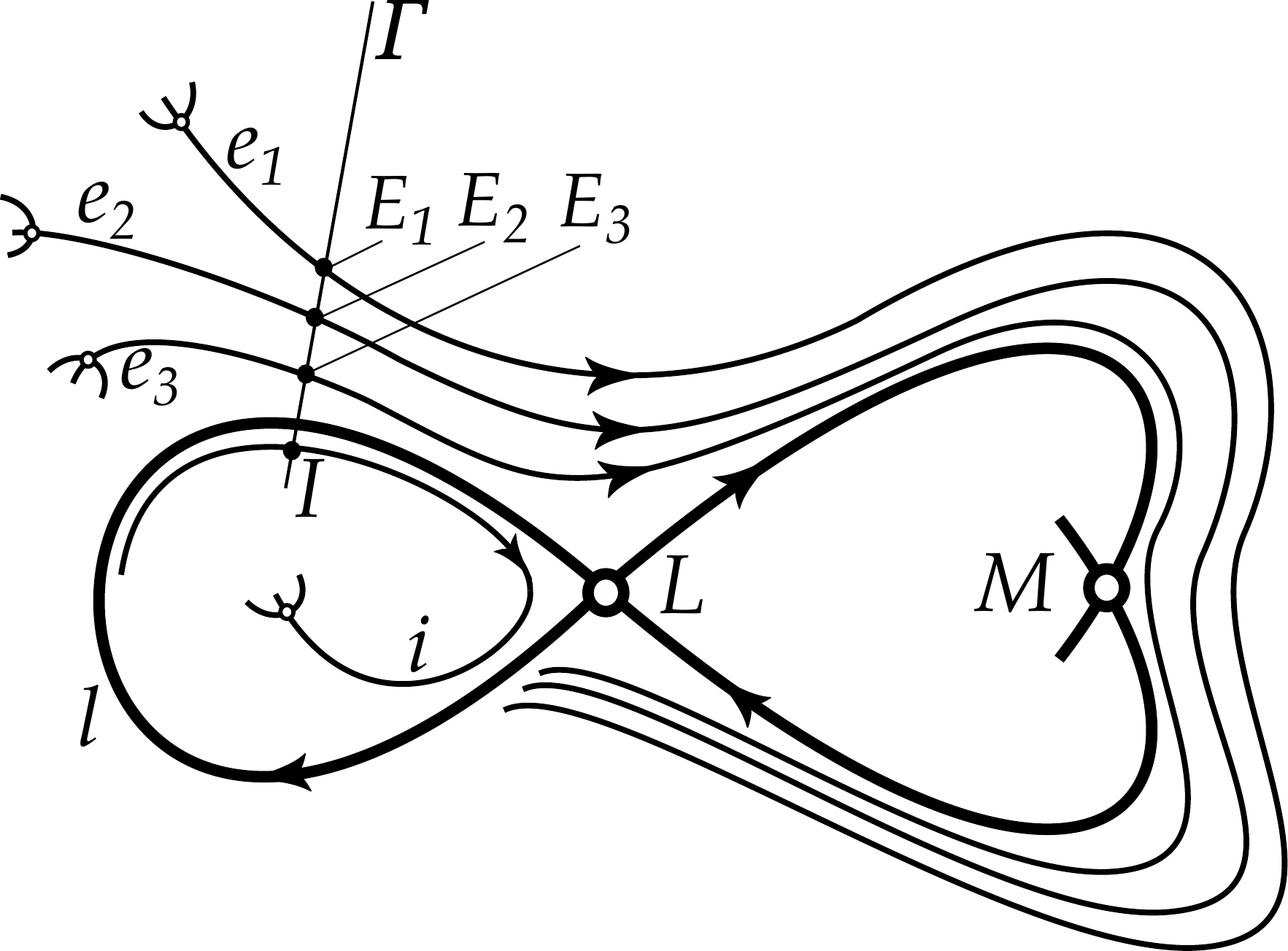}
    \caption{A vector field of class \({\mathbf{TH}}_3\)}
    \label{fig:pc-th3}
\end{figure}

Consider a normalized vector field \(v\in \Vect(S^2)\).
Suppose that it has 
\begin{itemize}
  \item two saddle points~\(L\) and~\(M\);
  \item two separatrix connections between \(L\) and \(M\);
  \item a~separatrix loop~\(l\) of~\(L\).
  \item no~saddle connections other than the three connections described above, and no non-hyperbolic singular points.
\end{itemize}
The polycycle~\(\gamma \) formed by the separatrix loop (the \emph{tear}) and the separatrix connections between \(L\) and \(M\) (the \emph{heart}) is called a \emph{polycycle of type~\(TH\)}, if the tear and the heart are located outside of each other, and the two yet unused separatrices of~\(M\) are located inside the heart.
These additional conditions imply that~\(\gamma \) is monodromic from exterior, and~\(l\) is monodromic from interior.

Now suppose that the characteristic numbers \(\lambda \), \(\mu \) of \(L\), \(M\) satisfy the inequalities
\begin{align}
  \label{eq:l-m-ineq}
  \Lambda _{i} \coloneqq  \lambda  &< 1; & \Lambda _{e}\coloneqq \lambda ^2\mu  &> 1.
\end{align}
Then the Poincaré map along~\(\gamma \) strongly attracts to the polycycle, while the Poincaré map along~\(l\) strongly repels from it, see~\cite[Sec. 4.3, Remark 12]{IKS-th1}.
Hence, the following assumptions do not increase the order of degeneracy.
\begin{definition}
    We say that a normalized vector field with a polycycle of~type~\(TH\) belongs to~\({\mathbf{TH}}_N\), if
    \begin{itemize}
      \item there is exactly one separatrix~\(i\) of some other saddle winding onto~\(l\) from the interior in the reverse time;
      \item there are exactly \(N\) separatrices~\(e_k\), \(k=1,\dots ,N\), winding onto~\(\gamma \) from the exterior;
      \item the separatrix~\(e_1\) is~topologically distinguished.
    \end{itemize}
\end{definition}
The last assumption means that for two orbitally topologically equivalent vector fields \(v, {\tilde v}\in {\mathbf{TH}}_N\), the homeomorphism linking their phase portraits brings the distinguished separatrix of \(v\) to the distinguished separatrix of~\({\tilde v}\).
There are many ways to achieve this that formally lead to different classes~\({\mathbf{TH}}_N\), and it's not important which one is used.

\cite[Theorem 4]{IKS-th1} states that
\begin{equation}
    \label{eq:phi}
    \varphi (v)=\frac{\ln \Lambda _{e}(v)}{-\ln \Lambda _{i}(v)}
\end{equation}
satisfies the first conclusion of~\autoref{thm:main-ex} if \(N=1\), and its proof can be trivially adjusted for the general case.

\subsection{Special coordinates on a cross-section}%
\label{sec:coords}
In order to define the map~\(\Phi \), we need to introduce special coordinates on a cross-section to the polycycle.
Let \(\Gamma \) be a cross-section to the tear~\(l\) at a non-singular point~\(O\).
Let \(\Gamma _e\) and \(\Gamma _i\) be the exterior and interior parts of~\(\Gamma \), respectively.
Let \(\Delta _e:(\Gamma _e, O)\to (\Gamma _e, O)\) and~\(\Delta _{i}:(\Gamma _{i}, O)\to (\Gamma _{i}, O)\) be the Poincaré maps along~\(\gamma \) and~\(l\), respectively.
Let~\(x:\Gamma \to {\mathbb R}\) be a smooth chart on~\(\Gamma \) positive on~\(\Gamma _{e}\) and negative on~\(\Gamma _{i}\)

The Poincaré map~\(\Delta _{e}\) behaves like~\(x\mapsto x^{\Lambda _e}\) near the origin, see~\cite[Corollary 3]{IKS-th1} or~\autoref{lem:IKS} below.
In the chart~\(\xi =\ln(-\ln |x|)\), the map~\(x\mapsto x^{\Lambda _e}\) is given~by~\(\xi \mapsto \xi +\ln \Lambda _{e}\), and the Poincaré map is close to it.
One can use this fact to construct a~\emph{rectifying chart}~\(\xi _e\) that conjugates~\(\Delta _{e}\) to the shift~by~\(\ln \Lambda _{e}\), and a similar chart~\(\xi _{i}\) for~\(\Delta _{i}\).
Namely, in~\autoref{sec:rectifying-chart} we shall prove the following lemma.
\begin{lemma}%
    \label{lem:rect-chart}
    There exist unique continuous maps~\(\xi _e:(\Gamma _e, O)\to ({\mathbb R}_{+}, +\infty )\), \(\xi _{i}:(\Gamma _{i}, O)\to ({\mathbb R}_{+}, +\infty )\) such that
    \begin{align*}
      \xi _{e}\circ \Delta _{e}&=\xi _{e}+\ln \Lambda _{e}; & \xi _{e}&=\ln(-\ln |x|)+o(1);\\
      \xi _{i}\circ \Delta _{i}&=\xi _{i}+\ln \Lambda _{i}; & \xi _{i}&=\ln(-\ln |x|)+o(1).
    \end{align*}
\end{lemma}
Note that \(\xi _e\mod (\ln \Lambda _e){\mathbb Z}\) is a~well-defined chart on~the quotient space~\(\Gamma _e/\Delta _e\).
One can think about the quotient space~\(\Gamma _{e}/\Delta _{e}\) as a closed curve without contact surrounding~\(\gamma \), or as a fundamental domain \((\Delta _{e}(x_0), x_0]\), \(x_0\in \Gamma _{e}\).
\subsection{The invariant}%
\label{sec:the-invariant}
The separatrices \(e_k\), \(k=1,\dots ,N\), intersect~\(\Gamma _e\) in~\(N\)~orbits of~\(\Delta _{e}\).
These orbits split the circle~\(\Gamma _{e}/\Delta _{e}\) into~\(N\) arcs.
Let~\(\Phi _{k}\) be the length of the \(k\)-th arc with respect to the chart~\(\xi _{e}\mod (\ln \Lambda _{e}){\mathbb Z}\), rescaled so that~\(\sum _{k=1}^{n}\Phi _{k}=\frac{\ln \Lambda _{e}}{-\ln \Lambda _{i}}\).

More precisely, let us reenumerate the separatrices \(e_2, \dots , e_N\) and choose intersection points \(E_k\in e_k\cap \Gamma _e\) so that \(x(E_1)>\cdots >x(E_N)>x(\Delta _e(E_1))\).
Put
\begin{align}
  \label{eq:Phi-k}
  E_{N+1}&\coloneqq \Delta _{e}(E_{1}),&\Phi _{k}(v)&\coloneqq \frac{\xi _{e}(E_{k+1})-\xi _{e}(E_{k})}{-\ln \Lambda _{i}}.
\end{align}
Then \(\varphi (v)=\sum  \Phi _k(v)=-\frac{\ln \Lambda _{e}}{\ln \Lambda _{i}}\) which agrees with~\eqref{eq:phi}.
Clearly, the numbers \(\Phi _k\) do not depend on the choice of \(E_k\).

Finally, we are ready to formulate the explicit version of~\autoref{thm:main-ex}.
\begin{theorem}%
    \label{thm:main}
    The class~\({\mathbf{TH}}_N\) and the map \(\Phi =(\Phi _1,\dots ,\Phi _N)\) constructed above satisfy the conclusions of~\autoref{thm:main-ex}.
\end{theorem}

\section{Proof of the Main Theorem modulo technical lemmas}%
\label{sec:proof}
\subsection{Unfolding of the “Tears of the heart” polycycle}%
\label{sub:unfolding}
Consider an unfolding~\(V=\set{v_{\alpha }}\in {\mathbf{TH}}_N^\pitchfork \) of a vector field \(v=v_0\) of class~\({\mathbf{TH}}_N\).
Let (semi-)transversals \(\Gamma \), \(\Gamma _i\), \(\Gamma _e\), and a coordinate \(x:\Gamma \to {\mathbb R}\) be as in \autoref{sec:coords}.

For \(\alpha \) close to zero, let \(L(\alpha )\) and \(M(\alpha )\) be the saddles close to \(L\) and \(M\), respectively.
The saddle~\(L(\alpha )\) has a stable separatrix \(s(\alpha )\) and an unstable separatrix~\(u(\alpha )\) continuously depending on~\(\alpha \) such that both \(s(0)\) and \(u(0)\) are the loop~\(l\).
Let \(S(\alpha )\) be the first intersection point of \(s(\alpha )\) with \(\Gamma \), counting from~\(L(\alpha )\).
Let \(U(\alpha )\) be the first intersection point of \(u(\alpha )\) with \(\Gamma \), counting from~\(L(\alpha )\).
The \emph{separatrix splitting} parameter
\begin{equation}
    \label{eq:eps}
    \eps \coloneqq x(S(\alpha ))-x(U(\alpha ))
\end{equation}
measures how far \(v_\alpha \) is from having a loop close to \(l\).
More precisely, \(v_\alpha \) has a loop close to~\(l\) if and only if \(\eps =0\).

We can introduce similar separatrix splitting parameters~\(\sigma _2\), \(\sigma _3\) for the separatrix connections between~\(L\) and \(M\), and the triple \(\sigma \coloneqq (\eps , \sigma _2, \sigma _3)\) is a coordinate system on the base of the family, see \cite[Sec. 2.2]{IKS-th1} for details.

As~in~\cite[Sec. 2.2.4]{IKS-th1}, let \(({\mathcal E}, 0)\) be given by \(\sigma _2=\sigma _3=0\).
Geometrically, \({\mathcal E}\) is the set of parameters \(\alpha \in ({\mathbb R}^3, 0)\) such that the saddles \(L(\alpha )\) and \(M(\alpha )\) have two separatrix connections close to those of \(L\) and \(M\).
This one-dimensional subfamily is parametrized by the parameter~\(\eps \) introduced above.
Let \({\mathcal E}_+\subset {\mathcal E}\) be given by \(\eps >0\).

From now on (unless stated otherwise) we deal only with the subfamily~\(\set{v_\alpha |\alpha \in {\mathcal E}}\), and use~\(\eps \) as~a~parameter on this subfamily.
\subsection{Sparkling separatrix connections}%
\label{sec:sparkl-saddle-conn}
Let \(\eps \), \({\mathcal E}_+\), \(L(\eps )\), \(M(\eps )\), \(s(\eps )\), \(u(\eps )\), \(S(\eps )\) and \(U(\eps )\) be as in~\autoref{sub:unfolding}.
Let \(i\), \(e_{1}\), \dots , \(e_{N}\) be the separatrices of~\(v\) introduced in \autoref{sec:polycycle-tear-heart}, enumerated as in \autoref{sec:the-invariant}.
For small \(\eps \), the vector field~\(v_{\eps }\) has separatrices \(i(\eps )\), \(e_{1}(\eps )\), \dots , \(e_{N}(\eps )\) continuously depending on~\(\eps \) such that \(i(0)=i\), \(e_{k}(0)=e_{k}\).
Let \(E_k(\eps )\), \(k=1,\dots ,N\), be continuous families of intersection points \(E_k(\eps )\in e_k(\eps )\cap \Gamma _e\) such that \(E_k(0)=E_k\).
Similarly, let us fix an intersection point \(I\in i\cap \Gamma _i\), and choose a continuous family \(I(\eps )\in i(\eps )\cap \Gamma _i\).

For some small positive values of \(\eps \), the separatrices \(u(\eps )\) and \(s(\eps )\) form the following “sparkling” separatrix connections.
\begin{itemize}
  \item for \(\eps =\iota _n\), the separatrix \(u(\eps )\) makes \(n\) turns around the loop~\(l\), then comes to~\(I(\eps )\);
  \item for \(\eps =\eps _{k,m}\), the separatrix \(s(\eps )\) makes \(m\) turns around the polycycle~\(\gamma \) in backward time, then comes to~\(E_k(\eps )\).
\end{itemize}
One can show that \(\iota _{n}\searrow 0\) as \(n\to \infty \) and \(\eps _{k,m}\searrow 0\) as \(m\to \infty \), see \cite[Lemma 1]{IKS-th1}.
\begin{remark}%
    \label{rem:num-turns}
    The number of turns the separatrix~\(s(\eps )\) makes around the polycycle~\(\gamma \) before coming to~\(E_{k}(\eps )\)) is not defined by the phase portrait of~\(v_{\eps }\).
    Indeed, a \wiki{Dehn twist} along a curve surrounding~\(\gamma \) changes this number by one.

    However, in a \emph{family} of vector fields, one can define the number of turns simultaneously for all~\(v_{\eps }\) up to an additive constant.
    One of the ways to do this is discussed in~\cite[Definition 3]{IKS-th1}, but it works only if there is exactly one separatrix outside of~\(\gamma \), and one inside~\(l\).
    Another way is to fix a~cross-section~\(\Gamma \) and points~\(E_{k}\) as above, then say that the number of turns is the number of intersection points \(s(\eps )\cap \Gamma \) between~\(S(\eps )\) and~\(E_{k}(\eps )\), including one of these two points, and similarly for the interior separatrix connections.
\end{remark}

\subsection{Plan of the proof}%
\label{sec:plan-proof}
One can prove (see~\cite[Lemma 1]{IKS-th1}) that for \(n\) large enough we have \(\iota _n>\iota _{n+1}>\cdots \), and \(\forall k, \eps _{k,m}>\eps _{k,m+1}>\cdots \).
Actually, the same arguments~\cite[Sec. 4.6.2]{IKS-th1} imply that
\begin{equation}
    \label{eq:ekm-order}
    \eps _{1,m}>\eps _{2,m}>\cdots >\eps _{N,m}>\eps _{1,m+1}>\cdots 
\end{equation}
However, the numbers \(\iota _n\) are interspered between the numbers \(\eps _{k,m}\) in a non-trivial way.
In particular, the \emph{relative density} depends on the ratio of~\(\ln \Lambda _{i}\) and \(\ln \Lambda _{e}\), see~\cite[Corollary 1]{IKS-th1}:
\[
    \lim_{\eps \to 0+}\frac{\#\set{n|\iota _n>\eps }}{\#\set{m|\eps _{1,m}>\eps }}=\frac{\ln \Lambda _{e}}{-\ln \Lambda _{i}}=\varphi (v).
\]
This fact was used in~\cite{IKS-th1} to prove that the right hand side is an invariant of moderate topological equivalence of~families \(V\in {\mathbf{TH}}_1^\pitchfork \).

For each \(n\), let \(\eps _{k_n,m_n}\) be the smallest number of the form~\(\eps _{k,m}\) that is greater than~\(\iota _n\).
Consider the frequencies
\begin{equation}
    \label{eq:freq-def}
    \psi _{\eps ,k}\coloneqq \frac{\#\set{n|\iota _n>\eps , k_n=k}}{\#\set{n|\iota _n>\eps }}.
\end{equation}
In~\autoref{sec:estim-in-ekm}, \autoref{lem:estim-ekm-in} we state some estimates on~\(\iota _{n}\) and~\(\eps _{k,m}\), and postpone their proofs to~\autoref{sec:asympt-sparkl-saddle}.
These estimates imply that on the circle \({\mathbb R}/(\ln \Lambda _{e}){\mathbb Z}\), each sequence~\(m\mapsto \ln(-\ln \eps _{k,m})\) converges to~\(\xi _{e}(E_{k})\), while \(\ln(-\ln \iota _{n})\) is close to an orbit of a rotation through~\((-\ln \Lambda _{i})\).

The frequencies~\(\psi _{\eps ,k}\) describe how often the sequence \(\ln(-\ln \iota _{n})\) visits each of~\(N\)~arcs of the circle~\({\mathbb R}/(\ln \Lambda _{e}){\mathbb Z}\).
If the rotation angle and the circle length are incommensurable, then the orbits are uniformly distributed on the circle, so these frequencies are equal to the normalized lengths~\(\frac{\Phi _{k}(v)}{\varphi (v)}\);
otherwise, we can estimate the difference between these two values, see \autoref{cor:freq} in \autoref{sec:sequ-close-orbits} for details.
Finally, in~\autoref{sec:two-equiv-fam} we use these estimates to prove \autoref{thm:main}.

\subsection{Estimates on \(\iota _n\), \(\eps _{k,m}\)}%
\label{sec:estim-in-ekm}

In order to estimate the frequencies~\(\psi _{\eps ,k}\), see \eqref{eq:freq-def}, we shall first estimate \(\iota _n\) and \(\eps _{k,m}\).
Recall that~\(\Delta _{e}\) and~\(\Delta _{i}\) are the Poincaré maps along the polycycle~\(\gamma \) and the loop~\(l\), respectively.
For \(\eps \) close to zero, we can consider the Poincaré maps~\(\Delta _{e,\eps }, \Delta _{i,\eps }\) along the broken polycycle~\(\gamma \) and the broken loop~\(l\).
Then \(\iota _n\) and~\(\eps _{k,m}\) are the unique roots of the equations
\begin{align}%
  \label{eq:in-ekm}
  \Delta _{i,\eps }^{n}(U(\eps ))&=I(\eps ) & \Delta _{e,\eps }^{-m}(S(\eps ))=E_{k}(\eps ).
\end{align}
In~\autoref{sec:asympt-sparkl-saddle}, we shall prove that in appropriate charts, the Poincaré maps~\(\Delta _{i,\eps }, \Delta _{e,\eps }\) are close to the maps~\(\Delta _{i}, \Delta _{e}\), and use this fact to estimate~\(\iota _{n}\) and~\(\eps _{k,m}\).
Namely, we shall prove the following lemma.
\begin{lemma}%
    \label{lem:estim-ekm-in}
    In the settings introduced above, we have
    \begin{align}
      \label{eq:estim-in}
      \ln(-\ln \iota _{n})&=-n\ln \Lambda _{i}+\xi _i(I)+o(1),\\
      \label{eq:estim-ekm}
      \ln(-\ln \eps _{k,m})&=m\ln \Lambda _{e}+\xi _e(E_{k})+o(1),
    \end{align}
    where~\(\xi _{e}\), \(\xi _{i}\) are the rectifying charts introduced in~\autoref{sec:coords}.
\end{lemma}
Though we postpone the rigorous proof of this lemma, let us explain the main idea of the proof right here.
\begin{proof}%
    [Idea of the proof of \autoref{lem:estim-ekm-in}]
    We shall only discuss the proof of~\eqref{eq:estim-in} here.
    Consider the chart~\(x_{\eps }=x(S(\eps ))-x\).
    Due~to~\eqref{eq:eps}, we have~\(x_{\eps }(U(\eps ))=\eps \).
    It turns out that for \(\eps \) close to~zero, the map~\(\Delta _{i,\eps }\) written in the chart~\(x_{\eps }\) is very close to~the~map~\(\Delta _{i}\) written in the chart~\(x_{0}=-x\), as long as \(x_{\eps }\geq \eps \).
    So, we can replace this equation with the equation
    \[
        \Delta _{i}^{n}(\iota _{n})=I.
    \]
    Now, rewriting this equation in the rectifying chart~\(\xi _{i}\) we get
    \[
        \ln(-\ln \iota _{n})+n\ln \Lambda _{i}+o(1)=\xi _{i}(I).
    \]
    Moving some terms to the right hand side of the equation, we get~\eqref{eq:estim-in}.
\end{proof}
In particular, \eqref{eq:estim-ekm} implies that for \(m\) large enough, the numbers~\(\eps _{k,m}\) are indeed ordered as~in~\eqref{eq:ekm-order}.

\subsection{Sequences close to orbits of a rotation}%
\label{sec:sequ-close-orbits}
Let us estimate the frequencies~\eqref{eq:freq-def}.
First, we prove a general lemma.
\begin{lemma}%
    \label{lem:freq}
    Let~\((J_{n})\), \(J_n\subset S^{1}={\mathbb R}/{\mathbb Z}\) be a~converging sequence of intervals, either open, closed or half-open.
    Let~\(J=[a,b]\) be its limit.
    Let~\((x_n)\) be a sequence close to an orbit of a rotation,
    \begin{align*}
        x_n&=c+n\rho +o(1)\in S^{1}, & c, \rho &\in S^{1}.
    \end{align*}
    Consider the frequency of~\(n\) such that \(x_n\in J_n\),
    \[
        \psi _{n}\coloneqq \frac 1n\#\set{j|1\leq j\leq n, x_j\in J_j}.
    \]
    Then
    \begin{itemize}
      \item for an irrational~\(\rho \), the sequence~\((\psi _{n})\) tends to \(|J|\) as~\(n\to \infty \);
      \item for a rational \(\rho =\frac{p}{q}\), \(\gcd(p, q)=1\), we have
        \begin{equation}
            \label{eq:freq-rational}
            -\frac{1}{q}+\limsup_{n\to \infty }\psi _{n}\leq |J|\leq \frac{1}{q}+\liminf_{n\to \infty }\psi _{n}.
        \end{equation}
    \end{itemize}
\end{lemma}
\begin{proof}
    Given \(\delta >0\), for \(j\) large enough, \(J_{j}\) is close to~\(J\) and \(x_{j}\) is close to~\(c+j\rho \), hence
    \begin{align*}
      c+j\rho \in [a+\delta ,b-\delta ]&\Rightarrow x_{j}\in (a, b), & x_{j}\in [a, b]&\Rightarrow c+j\rho \in [a-\delta ,b+\delta ],
    \end{align*}
    thus
    \begin{subequations}%
        \label{eq:freq-nbhd}
        \begin{align}
          \lim_{n\to \infty }\frac{1}{n}\#\set{j|1\leq j\leq n, c+j\rho \in [a+\delta , b-\delta ]}&\leq \liminf_{n\to \infty }\psi _{n},\\
          \lim_{n\to \infty }\frac{1}{n}\#\set{j|1\leq j\leq n, c+j\rho \in [a-\delta , b+\delta ]}&\geq \limsup_{n\to \infty }\psi _{n}.
        \end{align}
    \end{subequations}

    In the case of an irrational~\(\rho \), the orbit \((c+j\rho )\) is uniformly distributed on the circle, hence~\eqref{eq:freq-nbhd} implies
    \[
        b-a-2\delta \leq \liminf_{n\to \infty }\psi _{n}\leq \limsup_{n\to \infty }\psi _{n}\leq b-a+2\delta .
    \]
    Since this holds for any~\(\delta >0\), the sequence~\(\psi _{n}\) converges to~\(|J|=b-a\).

    In the case of a rational~\(\rho \), the orbit~\(c+j\rho \) of the shift \(x\mapsto x+\rho \) is periodic with period~\(q\), hence it visits any interval with the frequency equal to the fraction of the points~\(c+\rho , \dots , c+q\rho \) that belong to this interval.
    Hence~\eqref{eq:freq-nbhd} takes the form
    \begin{align*}
      \frac{1}{q}\#\set{j|1\leq j\leq q, c+j\rho \in [a+\delta , b-\delta ]}&\leq \liminf_{n\to \infty }\psi _{n},\\
      \frac{1}{q}\#\set{j|1\leq j\leq q, c+j\rho \in [a-\delta , b+\delta ]}&\geq \limsup_{n\to \infty }\psi _{n}.
    \end{align*}
    Since this holds for any~\(\delta >0\), we have
    \begin{align*}
      \frac{1}{q}\#\set{j|1\leq j\leq q, c+j\rho \in (a, b)}&\leq \liminf_{n\to \infty }\psi _{n},\\
      \frac{1}{q}\#\set{j|1\leq j\leq q, c+j\rho \in [a, b]}&\geq \limsup_{n\to \infty }\psi _{n}.
    \end{align*}
    Finally, an interval of~length~\(l\) contains at least~\(ql-1\) and at most~\(ql+1\) of the points~\(c+\rho , \dots , c+q\rho \), hence we have~\eqref{eq:freq-rational}.
\end{proof}
Now, apply this lemma to our case.
\begin{corollary}%
    \label{cor:freq}
    If~\(\varphi (v_{0})\) is irrational, then
    \begin{equation}
        \label{eq:freq-TH:irrational}
        \lim_{\eps \to 0}\psi _{\eps ,k}=\frac{\Phi _{k}(v_{0})}{\varphi (v_{0})}.
    \end{equation}
    If~\(\varphi (v_{0})\) is a rational number with denominator~\(q\), then
    \begin{equation}
        \label{eq:freq-TH:rational}
        -\frac{1}{q}+\varphi (v_{0})\limsup_{\eps \to 0}\psi _{\eps ,k}\leq \Phi _{k}(v_{0})\leq \frac{1}{q}+\varphi (v_{0})\liminf_{\eps \to 0}\psi _{\eps ,k}.
    \end{equation}
\end{corollary}

\begin{proof}
    Note that
    \begin{align*}
      k_{n} & =k &  & \iff  & \iota _{n}             & \in \left[\eps _{k+1,m_{n}}, \eps _{k,m_{n}}\right) \\
      \shortintertext{thus}
      k_{n} & =k &  & \iff  & \ln(-\ln \iota _{n}) & \in \left(\ln(-\ln \eps _{k,m_{n}}), \ln(-\ln \eps _{k+1,m_{n}})\right].
    \end{align*}
    Due to \autoref{lem:estim-ekm-in}, we have
    \begin{align*}
      \frac{1}{\ln \Lambda _{e}}\ln(-\ln \iota _{n})&=n\frac{-\ln \Lambda _{i}}{\ln \Lambda _{e}}+\frac{\xi _{i}(I)}{\ln \Lambda _{e}}+o(1),\\
      \frac{1}{\ln \Lambda _{e}}\ln(-\ln \eps _{k,m_{n}})&=\frac{\xi _{e}(E_{k})}{\ln \Lambda _{e}}+o(1)\pmod {\mathbb Z},\\
      \frac{1}{\ln \Lambda _{e}}\ln(-\ln \eps _{k+1,m_{n}})&=\frac{\xi _{e}(E_{k+1})}{\ln \Lambda _{e}}+o(1)\pmod {\mathbb Z}.
    \end{align*}
    Application of \autoref{lem:freq} completes the proof.
\end{proof}

\subsection{Two equivalent families}%
\label{sec:two-equiv-fam}
Now we are ready to prove \autoref{thm:main}.
Consider two moderately topologically equivalent families \(V, {\tilde V}\in {\mathbf{TH}}_N\).
As was noted in~\autoref{sec:polycycle-tear-heart}, the arguments from~\cite{IKS-th1} imply that~\(\varphi (v)=\varphi ({\tilde v})\).
So, it is enough to prove that we have~\(\Phi (v)=\Phi ({\tilde v})\) for an irrational~\(\varphi (v)\), and \(\left|\Phi _{k}(v)-\Phi _{k}({\tilde v})\right|\leq \frac{2}{q}\) for a rational~\(\varphi (v)=\frac{p}{q}\).

We shall use the already introduced notation (e.g., \({\mathcal E}_+\), \(\eps \), \(\iota _n\), \(\eps _{k,m}\)) for objects corresponding to~\(V\), while for similar objects corresponding to~\({\tilde V}\) we shall add tilde above.
Let \(H:(\alpha , x)\mapsto (h(\alpha ), H_\alpha (x))\) be the map implementing the equivalence of~\(V\) and~\({\tilde V}\).

As in~\cite[Sec. 2.3.3]{IKS-th1}, \(h\) sends \(({\mathcal E}, 0)\) to \(({\tilde {\mathcal E}}, 0)\), \({\mathcal E}_+\) to \({\tilde {\mathcal E}}_+\), a germ of the sequence~\((\iota _n)\) to a germ of the sequence~\(({\tilde \iota }_n)\), \(h(\iota _{n})={\tilde \iota }_{n+a}\), \(a\in {\mathbb Z}\), and a germ of the sequence~\((\eps _{k,m})\) to a germ of the sequence~\(({\tilde \eps }_{k,m})\).

Since the separatrix~\(e_1\) is topologically distinguished for vector fields from~\({\mathbf{TH}}_{N}\), \autoref{rem:moderate-corr} implies that \(H_{\eps }\) sends \(e_{1}(\eps )\) to \({\tilde e}_{1}(h(\eps ))\).
Thus we can distinguish~\(\eps _{1,m}\) from~\(\eps _{k,m}\), \(k\neq 1\), hence \(h\) sends \((\eps _{1,m})\) to~\(({\tilde \eps }_{1,{\tilde m}})\).
Next, recall that both sequences~\(\eps _{k,m}\) and~\({\tilde \eps }_{k,{\tilde m}}\) are ordered in the same way, see~\eqref{eq:ekm-order}, hence~\(h\) sends each~\(\eps _{k,m}\) to~\({\tilde \eps }_{k,m+b}\), where \(b\) is an integer constant.

Recall that \(\iota _{n}\) belongs to~\(\left[\eps _{k_{n}+1,m_{n}}, \eps _{k_{n},m_{n}}\right)\).
Since \(h(\iota _{n})={\tilde \iota }_{n+a}\) and \(h(\eps _{k,m})={\tilde \eps }_{k,m+b}\), we have \({\tilde \iota }_{n+a}\in \left[{\tilde \eps }_{k_{n}+1,m_{n}+b}, {\tilde \eps }_{k_{n},m_{n}+b}\right)\).
Thus \({\tilde k}_{n+a}=k_{n}\), hence
\begin{align*}
  \liminf_{\eps \to 0}\psi _{\eps ,k}&=\liminf_{\eps \to 0}{\tilde \psi }_{\eps ,k}, & \limsup_{\eps \to 0}\psi _{\eps ,k}&=\limsup_{\eps \to 0}{\tilde \psi }_{\eps ,k}.
\end{align*}
Finally, application of~\autoref{cor:freq} to both families~\(V, {\tilde V}\) completes the proof of~\autoref{thm:main}.
\section{Asymptotics of sparkling saddle connections}%
\label{sec:asympt-sparkl-saddle}
\subsection{General settings}%
\label{sec:general-settings}
In order to deal with both cases simultaneously, let us introduce notation according to \autoref{tab:notation}.
\begin{table}[h]
    \centering
    \begin{tabular}{ccc}
      \toprule
      \multirow{2}{*}{Notation} & \multicolumn{2}{c}{Meaning}       \\\cmidrule{2-3}
                                & Case~\(\eps _{k,m}\) & Case~\(\iota _{n}\) \\\midrule
      \(n\)                     & \(m\)            & \(n\)          \\
      \(\eps _{n}\)                 & \(\eps _{k,m}\)      & \(\iota _{n}\)      \\
      \(\Delta _{\eps }\)                 & \(\Delta _{e,\eps }^{-1}\) & \(\Delta _{i,\eps }\)    \\
      \(\Lambda \)                     & \(\Lambda _{e}^{-1}\)   & \(\Lambda _{i}\)      \\
      \(x_{\eps }\)                 & \(x-x(U(\eps ))\)    & \(x(S(\eps ))-x\)  \\
      \(P(\eps )\)                  & \(E_{k}(\eps )\)     & \(I(\eps )\)       \\
      \bottomrule
    \end{tabular}
    \caption{Notation unifying the cases of interior and exterior separatrix connections}%
    \label{tab:notation}
\end{table}

In this notation, both equations~\eqref{eq:in-ekm} take the form
\begin{equation}
    \label{eq:in-gen}
    \Delta _{\eps }^{n}(\eps )=P(\eps ).
\end{equation}

Our estimates will be based on the following lemma.
This lemma is a simple corollary of~\cites[Lemma 6]{IKS-th1};
see also~\cite[Theoreme 1]{Mo90} for much stronger estimates in case of infinitely smooth families, and~\cite[Sec 9.3]{HW-ode} for an alternative proof of~\eqref{eq:IKS:Delta}.
\begin{lemma}%
    \label{lem:IKS}
    In both cases described above, we have~\(\Lambda <1\), and
    \begin{subequations}%
        \label{eq:IKS}
        \begin{align}%
          \label{eq:IKS:Delta}
          \Delta _\eps (x)        & =x^{\Lambda }e^{O(1)};   \\
          \label{eq:IKS:Dx-Delta}
          D_x\Delta _\eps (x)     & =x^{\Lambda -1}e^{O(1)}; \\
          \label{eq:IKS:Deps-Delta}
          D_{\eps }\Delta _{\eps }(x) & = 1+o(1);
        \end{align}
    \end{subequations}
    as \((\eps , x)\) tends to the origin inside the angle~\(x\geq \eps \geq 0\).
\end{lemma}
\begin{proof}
    The inequality~\(\Lambda <1\) follows from~\eqref{eq:l-m-ineq}.
    The estimates~\eqref{eq:IKS} were proved in~\cite[Proof of Lemma 6]{IKS-th1} for any monodromic polycycle with the product~\(\Lambda \) of characteristic numbers along the polycycle being less than one.

    More precisely,
    \begin{description}
      \item[\eqref{eq:IKS:Delta}] is exactly~\cite[Eq. 35a]{IKS-th1};
      \item[\eqref{eq:IKS:Dx-Delta}] immediately follows from the estimate \(D_{x}\Delta _{\eps }(x)=x^{\Lambda (\eps )-1}e^{O(1)}\) from the proof of~\cite[Lemma 6, Eq. 35b]{IKS-th1}, and the estimate~\(x^{\Lambda (\eps )}=x^{\Lambda (0)}e^{O(1)}\) from the proof of~\cite[Lemma 6, Eq. 35a]{IKS-th1};
      \item[\eqref{eq:IKS:Deps-Delta}] is stated explicitly at the top of the proof of~\cite[Lemma 6, Eq. 35c]{IKS-th1}.
    \end{description}
\end{proof}
Now \autoref{lem:rect-chart} follows from \autoref{lem:gen:rect-chart} below applied to~\(\Delta ^{-1}\), and \autoref{lem:estim-ekm-in} follows from \autoref{lem:gen:estim} below.
\begin{lemma}%
    \label{lem:gen:rect-chart}
    Consider a~map \(\Delta :(0, \delta )\to (0, +\infty )\) satisfying~\eqref{eq:IKS:Delta} with some~\(\Lambda >1\).
    Suppose that\footnote{For \(\delta \) small enough, this assumption immediately follows from~\eqref{eq:IKS:Delta}.} \(\Delta ^{n}(x)\to 0\) for all \(0<x<\delta \).
    Then there exists a unique continuous map~\(\xi :(0, \delta )\to (0, +\infty )\) such that
    \begin{itemize}
      \item \(\xi (\Delta (x))=\xi (x)+\ln \Lambda \);
      \item \(\xi (x)=\ln(-\ln x)+O\left(\frac{1}{-\ln x}\right)\) as~\(x\to 0+\).
    \end{itemize}
\end{lemma}
In case of infinitely smooth vector fields, a similar statement was proved in~\cite{DR90}.

\begin{lemma}%
    \label{lem:gen:estim}
    Consider a family of maps~\(\Delta _{\eps }:(0, \delta )\to (0, +\infty )\) satisfying the conclusions of~\autoref{lem:IKS}.
    Let~\(P\) be a smooth function defined in a small neighborhood of~\(0\).
    Suppose that \(P(0)\) belongs to the attraction basin of the origin with respect to~\(\Delta _{0}^{-1}\).
    Then for \(n\) large enough, the equation~\eqref{eq:in-gen} has a unique solution~\(\eps _{n}\), and
    \begin{equation}
        \label{eq:gen:en-estim}
        \ln(-\ln \eps _{n})=-n\ln \Lambda +\xi (P(0))+o(1)
    \end{equation}
    as \(n\to \infty \).
\end{lemma}
The fact that~\eqref{eq:in-gen} has a unique root was proved in~\cite{IKS-th1}, together with a weaker estimate on the root, so we shall only prove the stronger estimate~\eqref{eq:gen:en-estim}.
In \autoref{sec:iter-poinc-map}, we shall prove some estimates on~\(\Delta _{\eps }^{n}(x)\) and its derivatives, then use them in~\autoref{sec:dist-unpert-poinc} to prove that
\begin{equation}
    \label{eq:en-0-approx}
    \Delta _{0}^{n}(\eps _{n})\approx P(0).
\end{equation}
Finally, in \autoref{sec:rectifying-chart} we shall prove \autoref{lem:gen:rect-chart}, then in~\autoref{sec:sparkl-saddle-conn-est} use it and~\eqref{eq:en-0-approx} to prove~\eqref{eq:gen:en-estim}.
\subsection{Iterates of the Poincaré map}%
\label{sec:iter-poinc-map}
Let \(\Delta _{\eps }\) be a family of functions satisfying conclusions of \autoref{lem:IKS}.
We start with some estimates on~\(\Delta _{\eps }^{n}(x)\) and its derivatives.
First, let us fix numbers~\(\delta >0\), \(0<c<1\), \(C>1\) such that for \(0\leq \eps \leq x<\delta \), \(x>0\), we have
\begin{subequations}%
    \label{eq:est-Delta}
    \begin{gather}
      \label{eq:est-Delta:map}
      cx^{\Lambda }<\Delta _{\eps }(x)<Cx^{\Lambda };                                \\
      \label{eq:est-Delta:Dx}
      c\frac{\Delta _{\eps }(x)}{x}<D_{x}\Delta _{\eps }(x)<C\frac{\Delta _{\eps }(x)}{x}; \\
      \label{eq:est-Delta:Deps}
      \frac{1}{2}<D_{\eps }\Delta _{\eps }(x)<2.
    \end{gather}
\end{subequations}

\begin{lemma}%
    \label{lem:est-iter}
    In the settings introduced above, given a number~\(\Lambda '\in (\Lambda , 1)\), there exists \(\delta '>0\) such that the following holds.
    Consider \(0\leq \eps \leq x<\delta \), \(x>0\), and a~natural~\(n\) such that all the iterates~\(\Delta _{\eps }^{k}(x)\), \(k=0,\dots ,n\), are less than~\(\delta \).
    Then
    \begin{subequations}
        \begin{gather}
            \label{eq:est-iter:map}
            c^{\frac{1}{1-\Lambda }}x^{\Lambda ^{n}}<\Delta _{\eps }^{n}(x)<C^{\frac{1}{1-\Lambda }}x^{\Lambda ^{n}};\\
            \label{eq:est-iter:Dx}
            c^{n}\frac{\Delta _{\eps }^{n}(x)}{x}<D_{x}\Delta _{\eps }^{n}(x)<C^{n}\frac{\Delta _{\eps }^{n}(x)}{x}.\\
            \intertext{If additionally \(x<\delta '\), then}
            \label{eq:est-iter:Deps}
            0<D_{\eps }\Delta _{\eps }^{n}(x)<x^{-\Lambda '}.
        \end{gather}
    \end{subequations}
\end{lemma}
\begin{proof}
    The estimates~\eqref{eq:est-iter:map} on~\(\Delta _{\eps }(x)\) immediately follow from~\eqref{eq:est-Delta:map} and~\(\Lambda <1\) by induction.
    Similarly, the estimates~\eqref{eq:est-iter:Dx} on~\(D_{x}\Delta _{\eps }^{n}(x)\) follow from the chain rule and~\eqref{eq:est-Delta:Dx}.

    Let us prove~\eqref{eq:est-iter:Deps}.
    By the chain rule, we get
    \[
        D_{\eps }\left(\Delta _{\eps }^{n}(x)\right)=\sum _{j=0}^{n-1} \left[D_{y}\Delta _{\eps }^{j}(y)\right]_{y=\Delta _{\eps }^{n-j}(x)}\left[D_{\eps }\Delta _{\eps }(y)\right]_{y=\Delta _{\eps }^{n-j-1}(x)}.
    \]
    All terms of the sum are positive, hence the derivative is positive as well, so we have the lower estimate.

    Now substitute~\eqref{eq:est-iter:Dx} and~\eqref{eq:est-Delta:Deps},
    \begin{align}
      \notag
      D_{\eps }\left(\Delta _{\eps }^{n}(x)\right) & < \sum _{j=0}^{n-1}2C^{j}\frac{\Delta _{\eps }^{n}(x)}{\Delta _{\eps }^{n-j}(x)}. \\
      \notag
      \intertext{Note that \(C^{j}<C^{n}\), and \(\Delta _{\eps }^{n-j}(x)\geq \Delta _{\eps }(x)\), hence}
      D_{\eps }\left(\Delta _{\eps }^{n}(x)\right) & < 2nC^{n}\frac{\Delta _{\eps }^{n}(x)}{\Delta _{\eps }(x)}.                   \\
      \intertext{Due to~\eqref{eq:est-Delta:map} and~\(\Delta _{\eps }^{n}(x)<\delta \), this implies}
      \label{eq:iter-est:Deps:with-n}
      D_{\eps }\left(\Delta _{\eps }^{n}(x)\right) & < 2nC^{n}\delta c^{-1}x^{-\Lambda }.
    \end{align}
    Now we need to estimate~\(nC^{n}\).

    Note that in the chart \(\xi =\ln(-\ln x)\), we start at~\(\ln(-\ln x)\), then make \(n\) jumps of size close~to~\(\ln \Lambda \), and arrive to a number larger than~\(\ln(-\ln \delta )\).
    Therefore, the number of jumps cannot be greater than~\(O(\ln(-\ln x))\).
    Formally, \(\Delta _{\eps }(x)<\delta \) together with the lower estimate in~\eqref{eq:est-iter:map} imply
    \begin{gather*}
        n<\frac{\ln(-\ln x)}{-\ln \Lambda }+a,\\
        \shortintertext{where}
        a = \frac{1}{-\ln \Lambda }\ln\left(\frac{1}{1-\Lambda }\ln c+\ln \delta \right).
    \end{gather*}
    Thus
    \begin{gather*}
        nC^{n}<C^{a}(-\ln x)^{-\frac{\ln C}{\ln \Lambda }}\left(\frac{\ln(-\ln x)}{-\ln \Lambda }+a\right).
    \end{gather*}
    The right hand side is asymptotically smaller than any negative power of~\(x\) as~\(x\to 0\), hence for~\(x\) small enough we have
    \[
        nC^{n}<\frac{c}{2\delta }x^{\Lambda -\Lambda '}.
    \]
    This inequality together with \eqref{eq:iter-est:Deps:with-n} implies the upper estimate in~\eqref{eq:est-iter:Deps}.
\end{proof}
\subsection{Distance to the unperturbed Poincaré map}%
\label{sec:dist-unpert-poinc}
Let us prove that \(\Delta _{\eps }^{n}(x)\) is close to~\(\Delta _{0}^{n}(x)\).
Let \(\delta , c, C\) be as in \autoref{sec:iter-poinc-map},~\eqref{eq:est-Delta}.
Let \(\Lambda '\in (\Lambda , 1)\) and~\(\delta '>0\) be as in \autoref{lem:est-iter}.

The following lemma immediately follows from \eqref{eq:est-iter:Deps}, the assumption~\(\Delta _{\eps '}^{n}(x)<\delta \), and the Mean Value Theorem.
\begin{lemma}%
    \label{lem:est-iter:dist}
    In the settings inroduced above, consider \(0\leq \eps \leq x<\delta '\), \(x>0\), and a~natural~\(n\) such that all the iterates~\(\Delta _{\eps '}^{k}(x)\), \(k=0,\dots ,n\), \(0<\eps '<\eps \), are less than~\(\delta \).
    Then
    \begin{equation}
        \label{eq:est-iter:dist}
        0<\Delta _{\eps }^{n}(x)-\Delta _{0}^{n}(x)<\eps x^{-\Lambda '}.
    \end{equation}
\end{lemma}
Let us use this lemma to~prove~\eqref{eq:en-0-approx}.
\begin{lemma}
    Let~\(\eps _{n}\) be the root of~\eqref{eq:in-gen}.
    Then
    \begin{equation}
        \label{eq:en-vs-en0}
        \Delta _{0}^{n}(\eps _{n})=P(0)+O\left(\eps _{n}^{1-\Lambda '}\right)
    \end{equation}
    as \(n\to \infty \).
\end{lemma}
\begin{proof}
    Applying \autoref{lem:est-iter:dist} to \(\eps =x=\eps _{n}\), we get
    \[
        0<\Delta _{\eps _{n}}^{n}(\eps _{n})-\Delta _{0}^{n}(\eps _{n})<\eps _{n}^{1-\Lambda '}.
    \]
    Since \(\eps _{n}\) is the root of~\eqref{eq:in-gen}, this implies
    \[
        P(\eps _{n})-\eps _{n}^{1-\Lambda '}<\Delta _{0}^{n}(\eps _{n})<P(\eps _{n}).
    \]
    Finally, substituting \(P(\eps _{n})=P(0)+O(\eps _{n})\) and taking into account \(1-\Lambda '<1\), we get~\eqref{eq:en-vs-en0}.
\end{proof}
Now we have an estimate on~\(\eps _{n}\) that involves only~\(\Delta _{0}\) but not~\(\Delta _{\eps _{n}}\).
The rest of the proof will be done in the rectifying chart for~\(\Delta _{0}\).

\subsection{Rectifying chart: proof of \autoref{lem:gen:rect-chart}}%
\label{sec:rectifying-chart}
\subsubsection*{Definition and uniqueness}
First, assume that we already have~\(\xi \), and try to find a~formula for this map.
Since~\(\xi (\Delta (x))=\xi (x)+\ln \Lambda \), we have
\[
    \xi (x)=\xi (\Delta ^{n}(x))-n\ln \Lambda .
\]
Recall that \(\Delta ^{n}(x)\to 0\) as \(n\to \infty \), hence~\(\xi (\Delta ^{n}(x))=\ln(-\ln \Delta ^{n}(x))+o(1)\).
So, if a rectifying chart exists, then it is given by the formula
\begin{align}
  \label{eq:rect-chart:def}
  \xi (x) & =\lim_{n\to \infty }\xi _{n}(x), & \xi _{n}(x) & =\ln(-\ln \Delta ^{n}(x))-n\ln \Lambda .
\end{align}
This proves uniqueness of the rectifying chart.

On the other hand,
\begin{equation}
    \label{eq:xi-n-conj}
    \xi _{n}(\Delta (x))=\xi _{n+1}(x)+\ln \Lambda ,
\end{equation}
so if~\eqref{eq:rect-chart:def} converges, then it defines a~rectifying chart.
Now, let us prove that \((\xi _{n}(x))\) actually converges, and the limit is close to~\(\ln(-\ln x)\).

\subsubsection*{Convergence and continuity}
Note that
\begin{align}
  \notag
  \xi _{n+1}(x)-\xi _{n}(x) & = \ln(-\ln \Delta ^{n+1}(x))-\ln(-\ln \Delta ^{n}(x))-\ln \Lambda  \\
  \label{eq:xi-diff:eq}
                     & = \left[\ln(-\ln \Delta (y))-\ln(-\ln y)-\ln \Lambda \right]_{y=\Delta ^{n}(x)}.
\end{align}
In order to estimate the latter expression, let us rewrite~\eqref{eq:IKS:Delta} in the chart \(\ln(-\ln x)\).

Namely, applying \(\ln\) to~\eqref{eq:IKS:Delta}, we get
\begin{align*}
  \ln\left(\Delta (x)\right)        & = \Lambda \ln x + O(1),                                                              \\
  \shortintertext{hence}
  -\ln\left(\Delta (x)\right)       & = \Lambda (-\ln x)\left(1+O\left(\frac{1}{-\ln x}\right)\right),                    \\
  \shortintertext{thus}
  \ln(-\ln\left(\Delta (x)\right)) & = \ln \Lambda  + \ln(-\ln x) + \ln\left(1+O\left(\frac{1}{-\ln x}\right)\right), \\
                               & = \ln(-\ln x)+\ln \Lambda +O\left(\frac{1}{-\ln x}\right)
\end{align*}
as~\(x\to 0+\).

Similarly to~\autoref{sec:iter-poinc-map}, introduce \(\delta '\in (0, \delta )\) and \(c, C>0\) such that for~\(x\in (0, \delta ')\) we have~\eqref{eq:est-Delta:map} and
\begin{equation}
    \label{eq:Deltat-ineq}
    \left|\ln(-\ln(\Delta (x)))-\ln(-\ln x)-\ln \Lambda \right|\leq \frac{C}{-\ln x}.
\end{equation}
We also require \(-\ln \delta '>\frac{2C}{\ln \Lambda }\), so~\eqref{eq:Deltat-ineq} implies
\[
    \ln(-\ln(\Delta (x)))\geq \ln(-\ln x)+\frac{\ln \Lambda }{2},
\]
hence
\begin{equation}
    \label{eq:Deltat-iter-geom}
    -\ln(\Delta ^{n}(x))\geq -\sqrt{\Lambda }^{n}\ln x.
\end{equation}

Let us prove that \((\xi _{n})\) converges to a continuous function on~\((0, \delta ')\).
Due~to~\eqref{eq:xi-n-conj}, this will imply convergence to a continuous function on the attraction basin of~\(0\) which includes~\((0, \delta )\).

Due to~\eqref{eq:xi-diff:eq} and~\eqref{eq:Deltat-ineq}, we have
\[
    \left|\xi _{n+1}(x)-\xi _{n}(x)\right|\leq \frac{C}{-\ln \left(\Delta ^{n}(x)\right) }.
\]
Applying~\eqref{eq:Deltat-iter-geom}, we get
\begin{equation}%
    \label{eq:Abel-majorate}
    \left|\xi _{n+1}(x)-\xi _{n}(x)\right|\leq \frac{C}{-\sqrt{\Lambda }^{n}\ln x}.
\end{equation}

Finally, the series
\[
    \xi (x)=\xi _{0}(x)+\sum _{n=0}^{\infty }(\xi _{n+1}(x)-\xi _{n}(x))
\]
is majorated by an infinite geometric series with common ratio~\(\Lambda ^{-0.5}<1\), hence it converges uniformly to a continuous function.

\subsubsection*{Estimate}
Due~to~\eqref{eq:Abel-majorate}, we have
\[
    |\xi (x)-\xi _{0}(x)|\leq \sum _{n=0}^{\infty }\frac{C}{-\sqrt{\Lambda }^{n}\ln x}=\frac{C\sqrt{\Lambda }}{-(\sqrt{\Lambda }-1)\ln x}=O\left(\frac{1}{-\ln x}\right).
\]
Note that~\(\xi _{0}(x)=\ln(-\ln x)\), thus we proved the required estimate on~\(\xi (x)\).

\subsection{Sparkling saddle connections: estimate}%
\label{sec:sparkl-saddle-conn-est}
Let us complete the proof of~\autoref{lem:gen:estim}.
Due to \autoref{lem:est-iter:dist},
\begin{equation}
    \label{eq:en-vs-en0-o1}
    \Delta _{0}^{n}(\eps _{n})=P(0)+O\left(\eps _{n}^{1-\Lambda '}\right)=P(0)+o(1).
\end{equation}
Let \(\xi \) be the rectifying chart for~\(\Delta _{0}\) provided by~\autoref{lem:gen:rect-chart}.
Then
\begin{gather*}
    \xi \left(\Delta _{0}^{n}(\eps _{n})\right)=\xi \left(P(0)+o(1)\right),\\
    \shortintertext{hence}
    \xi (\eps _{n})+n\ln \Lambda =\xi (P(0))+o(1),\\
    \shortintertext{thus}
    \ln(-\ln \eps _{n})+o(1)+n\ln \Lambda =\xi (P(0))+o(1).
\end{gather*}
This completes the proof of~\autoref{lem:gen:estim}, hence the proof of~\autoref{thm:main}.
\section{Corollaries, ideas and conjectures}%
\label{sec:future}
\subsection{Adding more separatrices $I_j$}%
\label{sec:adding-more-separ}
Let \({\mathbf{TH}}_{N,K}\) be the class of vector fields similar to~\({\mathbf{TH}}_{N}\), but with~\(K\) separatrices~\(i_{j}\) winding onto~\(l\) from interior.
Similarly to~\eqref{eq:Phi-k}, for~\(v\in {\mathbf{TH}}_{N,K}\) put~\(\Phi _{k}(v)=\xi _e(E_{k+1})-\xi _{e}(E_{k})\), this time without rescaling, and~\(\Psi _{j}(v)=\xi _{i}(I_{j+1})-\xi _{i}(I_{j})\), where \(I_{j}\) are defined in the same way as~\(E_{k}\), but for the separatrices~\(i_{j}\) instead of~\(e_{k}\).
Let \([\Phi :\Psi ](v)\) be the vector~\((\Phi _{1}(v), \dots , \Phi _{N}(v), \Psi _{1}(v), \dots , \Psi _{K}(v))\) as an element of~\({\mathbb R}{\mathbb P}^{N+K-1}\).

The proof of \autoref{thm:main} can be easily adjusted to prove the following theorem.
\begin{theorem}
    Given two moderately topologically equivalent local families of vector fields~\(V, {\tilde V}\in {\mathbf{TH}}_{N,K}^{\pitchfork }\) with irrational~\(\frac{\ln \Lambda _{e}}{-\ln \Lambda _{i}}\), we have \([\Phi :\Psi ](v_{0})=[\Phi :\Psi ]({\tilde v}_{0})\).
\end{theorem}

\subsection{Families with more parameters}%
\label{sec:families-with-more}
The construction of~\({\mathbf{TH}}_{N}\) can easily be adjusted to provide similar result for \(d\)-parameter family, \(d>3\).
Indeed, it is enough to add \(d-3\) semi-stable limit cycles of multiplicity~\(2\) surrounding the whole picture, cf.~\cite[Sec. 3.7]{IKS-th1}.
Then for a generic \(d\)-parameter unfolding of this vector field, the subfamily defined by the condition “all semi-stable limit cycles are unbroken” is a family of class \({\mathbf{TH}}_{N}^{\pitchfork }\), hence we can apply \autoref{thm:main} to this subfamily.
So, we have the following theorem.
\begin{theorem}%
    \label{thm:dim-d}
    For each natural \(N\) and \(d\geq 3\), there exist a Banach submanifold \({\mathbf{TH}}_N^{d}\subset \Vect(S^2)\) of codimension~\(d\) and a smooth surjective function~\(\Phi :{\mathbf{TH}}_N^{d}\to {\mathbb R}_+^N\) such that for two moderately topologically equivalent families \(V, {\tilde V}\in \left({\mathbf{TH}}_N^{d}\right)^\pitchfork \) we have
    \begin{itemize}
      \item \(\varphi (v_0)=\varphi ({\tilde v}_0)\), where \(\varphi (v)=\Phi _1(v)+\cdots +\Phi _N(v)\);
      \item if \(\varphi (v_0)\) is irrational, then \(\Phi (v_0)=\Phi ({\tilde v}_0)\).
      \item if \(\varphi (v_0)\) is a rational number with denominator~\(q\), then for each~\(k=1,\dots ,N\) we have \(\left|\Phi _{k}(v_{0})-\Phi _{k}({\tilde v}_{0})\right|\leq \frac{2}{q}\).
    \end{itemize}
\end{theorem}

However, our attempt to use this construction to obtain a~\(4\)-parameter family with a \emph{functional} invariant failed.
Recall the trick used in~\cite{IKS-th1} to obtain functional invariants.
Consider a generic \(4\)-parameter unfolding~\(v_{\alpha }\) of a vector field~\(v\in {\mathbf{TH}}_{N}\).
It meets~\({\mathbf{TH}}_{N}\) at a \(1\)-parameter subfamily.
There are~\(N\) functions~\(\Phi _{j}\) defined at each point of this subfamily.
Suppose that a moderate topological equivalence of \(4\)-parameter equivalence preserves these functions, i.e., for two equivalent families~\(v_{\alpha }\), \({\tilde v}_{{\tilde \alpha }}\) we have~\(\Phi _{j}(v_{\alpha })=\Phi _{j}({\tilde v}_{h(\alpha )})\).
Then we would have~\(N\) different parametrizations of the same curve, or equivalently, a curve in~\({\mathbb R}^{N}\) defined up to a change of coordinates in the domain, but not in the codomain.

The main difficulty at this path is the following.
Let \(\eps , \sigma _{2}, \sigma _{3}, \eta \) be the parameters of~\(v_{\alpha }\), where the first three parameters are the same as~in~\autoref{sub:unfolding}, and~\(\eta \) is an additional parameter.
As before, we are only interested in the subfamily given by~\(\sigma _{2}=\sigma _{3}=0\), but now it is a \(2\)-parameter subfamily parametrized by~\((\eps , \eta )\).

Though we can apply \autoref{lem:freq} to subfamilies~\(\eta =\const\), the homeomorphism~\(h\) from \autoref{def-moderate-eq} may send these curves to some non-vertical curves.
In particular, many intersection points of the curves~\({\tilde \eps }={\tilde \eps }_{k,m}({\tilde \eta })\) and~\({\tilde \eps }={\tilde \iota }_{n}({\tilde \eta })\) may be located between the curve~\(h(\set{\eta =\eta _{0}})\) and the vertical line~\({\tilde \eta }=h_{\set{\eps =0}}(\eta _{0})\), so the curve~\(h(\set{\eta =\eta _{0}})\) and the corresponding vertical line may meet the curves~\({\tilde \eps }={\tilde \eps }_{k,m}({\tilde \eta })\) and~\({\tilde \eps }={\tilde \iota }_{n}({\tilde \eta })\) in a very different order.

Currently we think that the path described above leads nowhere.
Moreover, it seems that for two generic \(4\)-parameter unfoldings~\(V\), \({\tilde V}\) of the same vector field, the subfamilies given by \(\sigma _{2}=\sigma _{3}=0\) and \({\tilde \sigma }_{2}={\tilde \sigma }_{3}=0\) are moderately topologically equivalent.
This leads to the following conjecture.

\begin{conjecture}
    The ratio~\(\frac{\ln \Lambda _{e}(v_{0})}{-\ln \Lambda _{i}(v_{0})}\) is the only invariant of moderate topological equivalence of generic \(4\)-parameter unfoldings of generic vector fields~\(v\in {\mathbf{TH}}_{N}\).
    A generic \(4\)-parametric unfolding of~\(v\in {\mathbf{TH}}_{N}\) is versal, see definition in~\cite[Sec. I.1.1.5]{AAIS94}.
\end{conjecture}

\subsection{Enriched dynamics}%
\label{sec:enriched-dynamics}
This article was inspired by the following idea.
 
Fix two closed curves curves without contact, \(C_{e}\) surrounding~\(\gamma \), and \(C_{i}\) close to~\(l\).
For a small \(\eps >0\), the correspondence map~\({\mathcal P}_{\eps }:C_{e}\to C_{i}\) along the trajectories of~\(v_{\eps }\) is defined on \(C_{e}\setminus s(\eps )\), and can be extended to a homeomorphism \(C_{e}\to C_{i}\) by setting \({\mathcal P}_{\eps }({\overline S}(\eps ))={\overline U}(\eps )\), where \(\set{{\overline S}(\eps )}=s(\eps )\cap C_{e}\), \(\set{{\overline U}(\eps )}=u(\eps )\cap C_{i}\).
This family of homeomorphisms contains a lot of information about bifurcations in the subfamily \({\mathcal E}_{+}\).
In particular, the separatrix connections described in~\autoref{sec:sparkl-saddle-conn} can be alternatively described by equations \({\mathcal P}({\overline E}_{k}(\eps ))={\overline S}(\eps )\), \({\mathcal P}({\overline U}(\eps ))={\overline I}(\eps )\), where \(\set{{\overline E}_{k}(\eps )}=e_{k}(\eps )\cap C_{e}\), \(\set{{\overline I}(\eps )}=i(\eps )\cap C_{i}\), and not yet discussed separatrix connections between~\(e_{k}(\eps )\) and~\(i(\eps )\) are given by~\({\mathcal P}_{\eps }({\overline E}_{k}(\eps ))={\overline I}(\eps )\).

The idea was to study the behaviour of~\({\mathcal P}_{\eps }\) as~\(\eps \to 0\).
We call the set of limit points of~\({\mathcal P}_{\eps }\), \(\eps \to 0\), with respect to an appropriate topology in the space of maps \(C_{e}\to C_{i}\) the \emph{enriched dynamics} of the original vector field.
The term was introduced by J.~Hubbard in the context of study of possible limits of the filled-in Julia set of~\(z\mapsto z^{2}+c\) as~\(c\) approaches the Mandelbrot set.

For a generic one-parametric unfolding of a quasi-generic vector field with a semi-stable limit cycle, in appropriate coordinates a similar map is close to a rotation.
This was used in~\cite{MP,GIS-semistable} to fully describe the classifications of such unfoldings with respect to normal and weak topological equivalence.

This idea led us to the same invariants as in \autoref{thm:main}.
When the first draft of this article was written, we proved the same theorem by simpler arguments, and rewrote the article from scratch.

It turns out that for \(\eps \) small enough, the graph of \({\mathcal P}_{\eps }\) looks like the letter “\(L\)” with a rounded corner.
More precisely, \({\mathcal P}_{\eps }\) expands a small interval \(\left({\overline S}(\eps ), {\overline S}(\eps )+\frac{C}{-\ln \eps }\right)\) to an interval slightly shorter than the whole circle, and contracts the interval \(\left({\overline S}(\eps )+\frac{C}{-\ln \eps }, {\overline S}(\eps )+\ln \Lambda _{e}\right)\) to a very short interval.
We have a plan to use this fact together with more precise estimates on~\(\eps _{k,m}\) and~\(\iota _{n}\) to solve the following problem.
\begin{problem}
    Describe all invariants of generic families of class~\({\mathbf{TH}}_{N}\).
\end{problem}
We are almost sure that the tuple \((\xi _{e}(E_{1}), \dots ,\xi _{e}(E_{N}),\xi _{i}(I), \ln \Lambda _{i}, \ln \Lambda _{e})\) defined up to some simple equivalence relation is the full invariant of classification of the subfamilies~\(V_{{\mathcal E}}\) of generic families~\(V\in {\mathbf{TH}}_{N}\), and hope that the same holds for the classification of the families themselves.

\section*{Acknowledgements}
We are grateful to Yu. Ilyashenko for the statement of the problem, his constant interest and support.
Our deep thanks to J. Hubbard whose talk about enriched dynamics inspired this paper.

\printbibliography
\end{document}